\documentclass[12pt]{amsart}
\usepackage{amssymb,latexsym}
\usepackage{amsfonts}
\usepackage{amsmath}
\usepackage[colorlinks,linkcolor=blue,anchorcolor=blue,citecolor=blue]{hyperref}
\usepackage{algorithm}
\usepackage{enumerate}
\usepackage{algpseudocode}
\usepackage{verbatim}
\usepackage{graphicx}

\newcommand\F{{\mathbb F}}

\newcommand\Z{{\mathbb Z}}

\newcommand\A{{\mathbb A}}
\newcommand\N{{\mathbb N}}

\newtheorem{theorem}{Theorem}[section]
\newtheorem{lemma}[theorem]{Lemma}
\newtheorem{corollary}[theorem]{Corollary}
\newtheorem{proposition}[theorem]{Proposition}

\theoremstyle{definition}
\newtheorem{definition}[theorem]{Definition}

\numberwithin{equation}{section}

\newcommand\sign{\text{\rm sign}}
\newcommand\cP{{\mathcal P}}

\newcommand\cT{{\mathcal T}}

\begin{document}

\title{Polynomial analogue of the Smarandache function}

\author{Xiumei Li}

\address{School of Mathematical Sciences, Qufu Normal University, Qufu, 273165, China}
\email{lxiumei2013@qfnu.edu.cn}

\author{Min Sha}
\address{School of Mathematics and Statistics, University of New South Wales, Sydney, NSW 2052, Australia}
\email{shamin2010@gmail.com}

\keywords{Erd{\H o}s's problem, factorial, Smarandache function, polynomials over finite fields}

\subjclass[2010]{11T06, 11T55}




\begin{abstract}
In the integer case, the Smarandache function of a positive integer $n$ is defined to be
the smallest positive integer $k$ such that $n$ divides the factorial $k!$.
In this paper, we first define a natural order for polynomials in $\F_q[t]$ over a finite field $\F_q$
and then define the Smarandache function of a non-zero polynomial $f \in \F_q[t]$, denoted by $S(f)$,
  to be the smallest polynomial $g$ such that $f$ divides the Carlitz factorial of $g$.
In particular, we establish an analogue of a problem of Erd{\H o}s, which implies that
for almost all polynomials $f$, $S(f)=t^d$, where $d$ is the maximal degree of the irreducible factors of $f$.
\end{abstract}

\maketitle

\section{Introduction}

\subsection{Motivation}

In number theory, the Smarandache function of a positive integer $n$
is defined to be the smallest positive integer $k$ such that $n$ divides the factorial $k!$.
This function was studied by Lucas \cite{Lucas} for powers of primes
and then by Neuberg \cite{Neu} and Kempner \cite{Kem} for general $n$.
In particular, Kempner \cite{Kem} gave the first correct algorithm for computing this function.
In 1980 Smarandache \cite{Sma} rediscovered this function.
It is also sometimes called the Kempner function.
This function arises here and there in number theory
(for instance, see \cite{Chen1995, Kempner1921,Luca,Sondow}).
Please see \cite{Liu} for a survey on recent results and \cite{HS} for a generalization to several variables.

Clearly, the Smarandache function of $n$ is equal to the maximum of those of its prime power factors.
For any integer $n \ge 2$, let $P(n)$ be the largest prime factor of $n$; and put $P(1)=1$.
For any $x > 1$, denote by $N(x)$ the number of
positive integers $n \le x$, whose Smarandache function is not equal to that of $P(n)$ (that is, $P(n)$, this means $n \nmid P(n)!$).
In 1991 Erd{\H o}s \cite{Erdos} posed a problem answered by Kastanas \cite{Kas} in 1994 that
 $N(x)=o(x)$ when $x$ goes to infinity.
 Later, Akbik \cite{Akbik} proved that $N(x) = O(x\exp(-\frac{1}{4}\sqrt{\log x}))$,
 and recently Ivi{\' c} \cite{Ivic} showed that
 $$
 N(x) = x \exp \left(-\sqrt{2\log x \log\log x} (1+ O(\log\log\log x /\log\log x)) \right);
 $$
 see \cite{DD,Ford} for some other previous results.
 Besides, Ivi{\' c} \cite{Ivic} pointed out that the main result in \cite[Equation (1.3)]{DD} is not correct.

In this paper, we define and study the Smarandache function for polynomials over a finite field.
In particular, we establish an analogue of Erd{\H o}s's problem.

\subsection{Our consideration}

Let $\F_q$ be the finite field of $q$ elements, where $q$ is a power of a prime $p$.
Denote by $\A=\F_q[t]$ the polynomial ring of one variable over $\F_q$ and $\N$ the set
of non-negative integers. Let $\N^*$ be the set of positive integers.
For any non-zero $g \in \A$, we denote by $\sign(g)$ the leading coefficient of $g$
(which is also called the sign of $g$).

We write $\F_q=\{a_{0}=0,a_{1}=1, a_2, \ldots,a_{q-1}\}$ throughout the paper.
For any non-zero polynomial $f\in \A$ of degree $n$, $f$ can be uniquely written as
\begin{equation}\label{eq:f}
f=a_{i_0}+a_{i_1}t+\ldots+a_{i_n}t^{n},  \quad a_{i_n} \ne 0, \, 0\leq i_j \leq q-1.
\end{equation}
We define $\delta(f)$ to be the integer:
\begin{equation*}\label{eq:delta}
\delta(f) = i_0 + i_1q + \cdots + i_n q^n.
\end{equation*}
In addition, we put $\delta(0)=0$.

Clearly, $\delta$ is a bijective map from $\A$ to $\N$, and for any $m\in\N$,
$$
\delta^{-1}(m)=a_{i_0}+a_{i_1}t+\ldots+a_{i_k}t^{k},
$$
 where $i_0 + i_1q + \cdots + i_k q^k$ is the $q$-adic expansion of $m$.

Moreover, we define an order in $\A$ based on the map $\delta$: for any $f,g\in \A$,
$$
f>g \quad \mbox{if and only if}\quad \delta(f) > \delta(g);
$$
and then $f \ge g$ if and only if $f > g$ or $f=g$.

With these preparations, we define a factorial in $\A$.

\begin{definition}\label{def:factorial}
For any non-zero polynomial $f \in \A$, the factorial of $f$ is defined to be
$$
f!=\prod_{g<f}(f-g).
$$
Additionally, we put $0!=1$.
\end{definition}

By definition, for any integer $n \ge 1$, $t^n !$ is in fact the product of all the monic polynomials of degree $n$.

This factorial is an analogue of the factorial of the rational integers;
see \cite{LS1} for another analogue.
It has been used in \cite{LS2,LS3}.
Notice that the above factorial of $f$ is equal to the multiplication of the Carlitz factorial of $f$ by a constant
(see the comment below Lemma~\ref{lem:factorial}).
For the Carlitz factorial, one can refer to \cite{Carlitz,Thakur}.

We now can define the Smarandache function for polynomials in $\A$.
In fact it has been used in \cite[Section 4.2]{LS2} for counting polynomial functions in the residue class rings
(see the definition of $\lambda$ there).

\begin{definition} \label{def:kem}
Given a non-zero polynomial $f \in \A$, the Smarandache function $S(f)$ of $f$ is defined to be the smallest polynomial $g$ such that
$$
f \mid g!.
$$
We put $S(0)=0$ by convention.
\end{definition}

In Section~\ref{sec:basic} we establish various basic properties of the Smarandache function $S$,
such as the computation, the value set, the inverse images, and fixed points.
We emphasize that several of them haven't been considered in the integer case,
such as Proposition~\ref{prop:delta-Kf} on how the size of a polynomial changes after an action of $S$
and Proposition~\ref{prop:dist} on the distance to fixed points.
We then in Section~\ref{sec:Erdos}  establish an analogue of Erd{\H o}s's problem for the function $S$ (see Theorem~\ref{thm:Erdos}).

\section{Preliminaries}

In this section, we gather some results which are used later on.

\subsection{Some elementary results}

We first compute the factorial $f!$ for any $f\in \A$.

\begin{lemma}\label{lem:factorial}
For any $f\in \A$ of the form \eqref{eq:f}, we have
\begin{equation} \label{eq:factorial}
f!=\left(\prod_{j=0}^{n}a_{i_j}! \right) \prod_{j=1}^{n}\prod_{\substack{h \in\A, \, \deg h=j \\ \sign(h)=1}} h^{i_j}.
\end{equation}
\end{lemma}

\begin{proof}
Denote by $R$ the right hand side of \eqref{eq:factorial}.
We rewrite $R$ as
$$
R=\prod_{j=0}^{n}\prod_{k=0}^{i_{j}-1}\prod_{ \substack{h\in\A, \, \deg h =j \\ \sign(h)=a_{i_j}-a_{k}}} h
= \prod_{j=0}^{n}\prod_{k=0}^{i_{j}-1}\prod_{ \substack{h\in\A, \, \deg h =j \\ \sign(h)=a_{i_j}-a_{k}}}(f-(f-h)),
$$
where one can see that $f-h$ exactly runs over all the polynomials $g < f$.
So, by definition we have $R=f!$.
\end{proof}

By definition and Lemma~\ref{lem:factorial}, $f!/ (\prod_{j=0}^{n}a_{i_j}!)$ is exactly the Carlitz factorial of $f$.
So, in Definition~\ref{def:kem} we can replace $g!$ by the Carlitz factorial of $g$.

The following result is a special case of Example 3 in \cite{Bhargava1997}.
We give a proof here.

\begin{lemma}\label{lem:val}
Let  $P \in \A$ be an irreducible polynomial of degree $d \ge 1$.
Then, for any non-zero polynomial $f\in \A$ we have
 \begin{equation*}
v_P(f!)=\sum_{j\geq 1}\left\lfloor\frac{\delta(f)}{q^{dj}}\right\rfloor,
\end{equation*}
where $v_P$ is the usual $P$-adic valuation.
\end{lemma}

\begin{proof}
Assume that $f$ is of the form \eqref{eq:f}.
From the formula \eqref{eq:factorial} of $f!$, we see that
for any integer $j \ge 1$, if $n= \deg f \ge dj$, then the number of terms in the right hand side of \eqref{eq:factorial}
divisible by $P^j$ is exactly equal to
$$
i_{dj} + i_{dj+1}q + \cdots + i_n q^{n-dj}.
$$
Summing up all these estimates we obtain the desired formula.
\end{proof}

Clearly, Lemma~\ref{lem:val} gives the following result.

\begin{corollary}\label{cor:div}
 For any $f,g\in \A$, if $g \le f$, then $ g! \mid f!$.
\end{corollary}

We remark that in Corollary~\ref{cor:div} the converse is not true.

For the proof of Proposition~\ref{prop:comp2}, we need the following lemma,
which is in fact a simple generalization of the formula of $\alpha$ in \cite[page 207]{Kem}
(also the formula in \cite[Lemma 1]{Sma}).

\begin{lemma}\label{lem:rep}
  Fix a positive integer $n>1$, and define a sequence $\{b_{j}=\frac{n^{j}-1}{n-1}: \, j\in\N^*\}$.
   Then for any $e\in\N^*$,  $e$ can be uniquely represented as
 \begin{equation*}\label{eq:rep}
e=c_{1}b_{j_1}+c_{2}b_{j_2}+\cdots+c_{k}b_{j_k},
\end{equation*}
where
$ j_1>j_2>\cdots>j_k>0 $
and $1\leq c_i < n$ for $i=1,2,\ldots,k-1, 1\leq c_k\leq n$.
\end{lemma}

\begin{proof}
Obviously, $\N^*$ is the disjoint union of the sets $[b_{j},b_{j+1})\cap \N^*, j\in\N^*$,
and $b_{j+1} = n b_j +1$ for any $j \in \N^*$.
 So, for any $e\in\N^*$, there exists an unique integer $j_{1}\in\N^*$ such that $e\in[b_{j_{1}},b_{j_{1}+1})\cap \N^* $,
 then by the division algorithm, we have
$$
e=c_{1}b_{j_{1}}+r_{1},
$$
 where $1\leq c_{1}=\lfloor\frac{e}{b_{j_{1}}}\rfloor\leq n$ and $0\leq r_{1}<b_{j_{1}}$.
If $r_{1}=0$, as $b_{j_{1}}\leq e <b_{j_{1}+1}$, then $k=1, 1\leq c_{1} \leq n$ and Lemma \ref{lem:rep}
is proved.

If $r_{1}\neq0$, as $b_{j_{1}}\leq e <b_{j_{1}+1}$, then $1\leq c_{1} < n$. The following procedure is the iterative process
 that makes use of the division algorithm in the form:
\begin{align*}
& r_{1}=c_{2}b_{j_{2}}+r_{2}, \quad 1\leq c_{2} < n, \,  0< r_{2}<b_{j_{2}}, \\
& r_{2}=c_{3}b_{j_{3}}+r_{3}, \quad 1\leq c_{3} < n, \, 0< r_{3}<b_{j_{3}}, \\
& \vdots \\
& r_{k-2}=c_{k-1}b_{j_{k-1}}+r_{k-1}, \quad 1\leq c_{k-1} < n, \, 0< r_{k-1}<b_{j_{k-1}}, \\
& r_{k-1}=c_{k}b_{j_{k}}, \quad 1\leq c_{k}\leq n.
\end{align*}
In the above computation the integer $k$ is defined by the condition that $r_{k-1} \neq 0$ and that $r_{k} = 0$.
Since $ e\geq b_{j_{1}}> r_{1}\geq b_{j_{2}}>r_{2}\geq \cdots \geq 0$,  such a $k$ must exist and $j_{1}>j_{2}>\cdots>j_{k}> 0$.

Collecting all the equalities above, Lemma \ref{lem:rep} is proved.
\end{proof}

\subsection{Counting polynomials}
\label{sec:count}

For any non-zero $f \in \A$, let $\omega(f)$ be the number of distinct monic irreducible factors of $f$,
and let $\tau(f)$ be the number of distinct monic factors of $f$.

The following two results should be well-known.

\begin{lemma} \label{lem:count-irre}
For any integer $n \ge 1$, the number of monic irreducible polynomials in $\A$ of degree at most $n$ is at most $q^n$.
\end{lemma}

\begin{proof}
For any monic irreducible polynomial $f \in \A=\F_q[t]$, if $f$ is of degree $d \le n$,
then $f$ corresponds to the monic polynomial $t^rf^s$ of degree $n$, where $n = sd + r$ with $0 \le r < d$ by the division algorithm.
Note that this corresponding is injective.
So the desired result follows.
\end{proof}

\begin{lemma} \label{lem:tau}
For any integer $n \ge 1$, we have
$$
\sum_{\substack{\textrm{monic $f \in \A$} \\ \deg f = n}} \tau(f) = (n+1)q^n.
$$
\end{lemma}

\begin{proof}
This result has been recorded in \cite[Proposition 2.5]{Rosen}.
Here we present a different proof.
It is easy to see that
\begin{align*}
\sum_{\substack{\textrm{monic $f \in \A$} \\ \deg f = n}} \tau(f)
& = \sum_{\substack{\textrm{monic $g \in \A$} \\ \deg g \le n}} \sum_{\substack{\textrm{monic $h \in \A$} \\ \deg h = n - \deg g}} 1
      = \sum_{\substack{\textrm{monic $g \in \A$} \\ \deg g \le n}} q^{n - \deg g}  \\
& = q^n \sum_{j=0}^{n} q^{-j} \cdot q^j = (n+1)q^n.
\end{align*}
\end{proof}

We now present some counting results for polynomials in $\A$ according to the numbers of their monic factors
and their maximal monic irreducible factors.
These are needed for proving Theorem~\ref{thm:Erdos}.

\begin{lemma} \label{lem:S1}
For any integer $n \ge 1$ and any real $r \ge 1$, let $B = 3r \log \log q^n$ and define
$$
\cT_1(n,r) = \{\textrm{monic }f \in \A:\, \deg f=n, \omega(f) > B\}.
$$
Then, we have
$$
|\cT_1(n,r)| < \frac{3q^n}{(\log q^n)^r}.
$$
If furthermore $n \ge 3$ and $r \ge 2$, we have
\begin{equation}  \label{eq:S1}
|\cT_1(n,r)| < \frac{q^n}{(\log q^n)^r}.
\end{equation}
Moreover, if $q \ge 3, n \ge 4$ and $r \ge 3$, in $\cT_1(n,r)$ we can choose $B = 2r\log \log q^n$,
and then the estimate~\eqref{eq:S1} still holds.
\end{lemma}

\begin{proof}
By definition, we have $\tau(f) \ge 2^{\omega(f)}$ for any non-zero $f \in \A$.
Using Lemma~\ref{lem:tau}, we deduce that
\begin{align*}
(n+1)q^n
&= \sum_{\substack{\textrm{monic $f \in \A$} \\ \deg f = n}} \tau(f)
  \ge \sum_{\substack{\textrm{monic $f \in \A$} \\ \deg f = n}} 2^{\omega(f)}
  \ge  \sum_{\textrm{$f \in \cT_1(n,r)$}} 2^{\omega(f)}  \\
& >  \sum_{\textrm{$f \in \cT_1(n,r)$}} 2^{3r\log \log q^n}
   =  2^{3r\log \log q^n} |\cT_1(n,r)|.
\end{align*}
So, we obtain
$$
|\cT_1(n,r)| < \frac{(n+1)q^n}{2^{3r\log \log q^n}}
= \frac{(n+1)q^n}{(\log q^n)^{3r\log 2}}
<  \frac{(n+1)q^n}{(\log q^n)^{r+1}}
< \frac{3q^n}{(\log q^n)^r}.
$$
If furthermore  $n \ge 3$ and $r \ge 2$, we have
$$
|\cT_1(n,r)| <  \frac{(n+1)q^n}{(\log q^n)^{3r\log 2}}
<  \frac{(n+1)q^n}{(\log q^n)^{2r}}
< \frac{q^n}{(\log q^n)^r},
$$
where the last inequality comes from
$$
(n \log q)^r \ge (n \log q)^2 \ge (n \log 2)^2 > n+1, \quad n \ge 3.
$$

The final part follows from
$$
|\cT_1(n,r)| <  \frac{(n+1)q^n}{(\log q^n)^{2r\log 2}}< \frac{q^n}{(\log q^n)^r}
$$
when $q \ge 3, n \ge 4$ and $r \ge 3$.
\end{proof}

\begin{lemma} \label{lem:S2}
For any integer $n \ge 1$ and any real $r \ge 1$, let $D=2r\log \log q^n$ and define
\begin{align*}
\cT_2(n,r) = \{& \textrm{monic }f \in \A:\, \deg f=n, \\
& P^2 \mid f \textrm{ for some irreducible polynomial  $P$ with $\deg P > D$}\}.
\end{align*}
Then, if $D \ge 4$, we have
\begin{equation}  \label{eq:S2}
|\cT_2(n,r)| <  \frac{q^n}{(\log q^n)^r}.
\end{equation}
Moreover, if $q \ge 3$, in $\cT_2(n,r)$ we can choose $D=r\log \log q^n \ge 12$,
and then the estimate~\eqref{eq:S2} still holds.
\end{lemma}

\begin{proof}
For any $f \in \cT_2(n,r)$, we can write $f=gP^2$ with $D < \deg P \le n/2$ and $\deg g = n - 2 \deg P$.
So, we have
\begin{align*}
|\cT_2(n,r)|
& \le \sum_{\substack{\textrm{monic irreducible $P \in \A$} \\ D < \deg P \le n/2}} \sum_{\substack{\textrm{monic $g \in \A$} \\ \deg g = n - 2\deg P}} 1 \\
& = \sum_{\substack{\textrm{monic irreducible $P \in \A$} \\ D < \deg P \le n/2}} q^{n - 2\deg P}
   < q^n \sum_{j=\lfloor D \rfloor+1}^{\infty} q^{-2j} \cdot q^j \\
& \le \frac{q^n}{q^{\lfloor D \rfloor}}  < \frac{q^n}{q^{-1+ 2r\log\log q^n}} = \frac{q^{n+1}}{(\log q^n)^{2r\log q}}  < \frac{q^n}{(\log q^n)^r},
\end{align*}
where the last inequality follows from $(\log q^n)^r \ge \exp(2)$ (due to $D \ge 4$), because it is equivalent to the following inequality
$$
q < (\log q^n)^{r(2\log q - 1)}.
$$

The second part follows similarly.
\end{proof}

\begin{lemma} \label{lem:S3}
For any integer $n \ge 1$ and any real $r \ge 1$, let $D=2r\log \log q^n$ and define
\begin{align*}
\cT_3 (n,r) = \{& \textrm{monic }f \in \A:\, \deg f=n, P^e \mid f, \\
&  e \ge D \textrm{ for some irreducible polynomial  $P$ with $\deg P \le D$}\}.
\end{align*}
Then, if $D \ge 8$, we have
\begin{equation}  \label{eq:S3}
|\cT_3(n,r)| < \frac{q^n}{(\log q^n)^r}.
\end{equation}
Moreover, if $q \ge 3$, in $\cT_3(n,r)$ we can choose $D=r\log \log q^n \ge 19$,
and then the estimate~\eqref{eq:S3} still holds.
\end{lemma}

\begin{proof}
Let $d = \lceil D \rceil \ge 8$.
By definition, for any $f \in \cT_3(n,r)$, there exists a monic irreducible polynomial  $P$ such that $\deg P \le D$ and $P^d \mid f$.
As in the proof of Lemma~\ref{lem:S2}, we have
\begin{align*}
|\cT_3(n,r)|
& \le \sum_{\substack{\textrm{monic irreducible $P \in \A$} \\ 1 \le \deg P \le D}} \sum_{\substack{\textrm{monic $g \in \A$} \\ \deg g = n - d\deg P}} 1 \\
& = \sum_{\substack{\textrm{monic irreducible $P \in \A$} \\ 1 \le \deg P \le D}} q^{n - d\deg P}
   < q^n \sum_{j=1}^{\infty} q^{-dj} \cdot q^j \\
& \le \frac{2q^n}{q^{D-1}}  = \frac{2q^{n+1}}{q^{2r\log \log q^n}} = \frac{2q^{n+1}}{(\log q^n)^{2r\log q}}  < \frac{q^n}{(\log q^n)^r},
\end{align*}
where the last inequality follows from $(\log q^n)^r \ge \exp(4)$ (due to $D \ge 8$),
because it is equivalent to the following inequality
$$
2q < (\log q^n)^{r(2\log q - 1)}.
$$

The second part follows similarly.
\end{proof}

\section{Basic properties}
\label{sec:basic}

\subsection{Computing the Smarandache function}

By definition, we directly obtain some simple properties about the Smarandache function.

\begin{proposition}   \label{prop:basic}
The following hold:
\begin{itemize}
\item[(1)] for any polynomial $f \in \A$ and any $a \in \F_q^*$,  $S(af)=S(f)$;

\item[(2)] for any non-zero polynomial $f \in \A$,  $S(f) \le t^{\deg f}$;

\item[(3)] for any irreducible polynomial $f \in \A$, $S(f) = t^{\deg f}$.
\end{itemize}
\end{proposition}

So, in order to compute the Smarandache function, we only need to consider monic polynomials.
By Definition \ref{def:kem} and Corollary \ref{cor:div}, we immediately obtain the following result,
which implies that we in fact only need to consider powers of monic irreducible polynomials.

\begin{proposition}\label{prop:comp1}
 Suppose that $P_1, P_2,\ldots, P_k$ are distinct monic irreducible polynomials
and $e_1,e_2,\ldots,e_k$ are positive integers. Then
 \begin{equation*}
 S(P_1^{e_1}P_2^{e_2}\cdots P_k^{e_k})=\max\{S(P_1^{e_1}),S(P_2^{e_2}),\ldots,S(P_k^{e_k})\}.
 \end{equation*}
\end{proposition}

The case of irreducible polynomials is straightforward.
We can in fact do more.

\begin{proposition}
Given a polynomial $f \in \A$ with $\deg f \ge 1$, assume that either $q \ge 3$ or  $f \ne b(t+c)^2$ for any $b,c \in \F_q$.
Then,  $f$ is an irreducible polynomial if and only if $S(f)=t^{\deg f}$.
\end{proposition}

\begin{proof}
We only need to prove the sufficiency.
Assume that $S(f)=t^{\deg f}$.
Without loss of generality, we can further assume that $f$ is monic.
By Proposition~\ref{prop:comp1}, we must have $f=P^e$ for some monic irreducible polynomial $P$ of degree $d$ and $e \ge 1$.

We first assume that $q \ge 3$.
If $e \ge 2$, since $a_2 P^{e-1}$ and $(a_2-1) P^{e-1}$ are two distinct terms in the factorial $(a_2 P^{e-1})!$ by definition,
 we have
$$
v_P((a_2 P^{e-1})!) \ge 2(e-1) \ge e,
$$
and so $S(P^e) \le a_2 P^{e-1} < t^{de}$, which contradicts with the assumption $S(P^e) = t^{de}$.
Thus, $f=P$ when $q \ge 3$.

We now assume that $q=2$.
By assumption, $f \ne (t+c)^2$ for any $c \in \F_q$.
So, if $d = 1$, we must have $e \ge 3$, and so $v_P(P^{e-1}!) \ge e$,
which implies $S(P^e) \le P^{e-1} < t^{e}$ and contradicts with the assumption $S(P^e) = t^{e}$.
Thus, we must have $d \ge 2$.
If $e \ge 2$, since $tP^{e-1}$ and $(t+1)P^{e-1}$ are two distinct terms in the factorial $t^{d(e-1)+1}!$ by definition, we have
$$
v_P(t^{d(e-1)+1}!) \ge 2(e-1) \ge e,
$$
and so $S(P^e) \le t^{d(e-1)+1} < t^{de}$, which contradicts with the assumption $S(P^e) = t^{de}$.
Thus, $f=P$. This completes the proof.
\end{proof}

We remark that in the case $q=2$, we have $S(t^2) = S(t^2+1) =  t^2$.

We now handle the case of powers of irreducible polynomials
 by following the strategy used to prove the theorem in \cite[page 208]{Kem}  (also \cite[Theorem 1]{Sma}).

\begin{proposition}\label{prop:comp2}
 Suppose that $P \in \A$ is an irreducible polynomial of degree $d \ge 1$ and $e$ is a positive integer.
 Define the sequence $b_{j}=\frac{q^{dj}-1}{q^{d}-1}, j\in\N^*$.
  Then, $e$ is uniquely written as
\begin{equation*}
e=c_{1} b_{j_1}+ c_{2} b_{j_2} + \cdots+c_{k} b_{j_k},
\end{equation*}
and
\begin{equation*}
S(P^e)= \delta^{-1}(c_{1}q^{dj_1}+c_{2}q^{dj_2}+\cdots+c_{k}q^{dj_k}),
\end{equation*}
where $ j_1>j_2>\cdots>j_k>0 $
and $1\leq c_i < q^{d}$ for $i = 1,2,\ldots,k-1, 1\leq c_k\leq q^{d}$.
\end{proposition}

\begin{proof}
By Lemma \ref{lem:rep}, we know that $e$ is uniquely written in the form:
\begin{equation*}
e=c_{1}b_{j_1}+c_{2}b_{j_2}+\cdots+c_{k}b_{j_k},
\end{equation*}
where $ j_1>j_2>\cdots>j_k>0 $
and $1\leq c_i < q^{d}$ for $i =1,2,\ldots,k-1, 1\leq c_k\leq q^{d}$.
Denote
\begin{align*}
m & =c_{1}q^{dj_1}+c_{2}q^{dj_2}+\cdots+c_{k}q^{dj_k} \\
& =(q^{d}-1)e+(c_{1}+c_{2}+\cdots + c_{k}).
\end{align*}
Since $\delta$  is a bijective map from $\A$ to $\N$,
we take $f=\delta^{-1}(m)$.
Then, it suffices to prove $S(P^e)=f$.

By Lemma~\ref{lem:val} and collecting the following equalities and inequalities
\begin{align*}
 & \left\lfloor\frac{m}{q^{d}}\right\rfloor = c_{1}q^{d(j_{1}-1)}+c_{2}q^{d(j_{2}-1)}+\cdots+c_{k}q^{d(j_{k}-1)}, \\
& \vdots \\
& \left\lfloor\frac{m}{q^{dj_k}}\right\rfloor = c_{1}q^{d(j_1-j_k)}+c_{2}q^{d(j_2-j_k)}+\cdots+c_{k},\\
 & \left\lfloor\frac{m}{q^{d(j_k+1)}}\right\rfloor
\geq
c_{1}q^{d(j_1-j_k-1)}+c_{2}q^{d(j_2-j_k-1)}+\cdots+c_{k-1}q^{d(j_{k-1}-j_{k}-1)}, \\
& \vdots \\
& \left\lfloor\frac{m}{q^{dj_1}}\right\rfloor\geq c_{1},
\end{align*}
 we obtain
\begin{equation*}
 v_P(f!) = \sum_{j \geq 1}\left\lfloor\frac{m}{q^{dj}}\right\rfloor \geq e,
\end{equation*}
which implies that $P^e \mid f!$. Actually,
 $v_P(f!) = e$ if and only if $c_{k} < q^{d}$.

Now, it remains to prove that for any $g\in\A$ and $g<f$, we have $P^e\nmid g!$.
In fact, by Corollary~\ref{cor:div},
we only need to prove that for $g = \delta^{-1}(m-1)$,  $P^e\nmid g!$;
that is, $v_P(g!)<e$. It is easy to obtain the following equalities:
\begin{align*}
 &\left\lfloor\frac{m-1}{q^{d}}\right\rfloor=c_{1}q^{d(j_{1}-1)}+c_{2}q^{d(j_{2}-1)}+\cdots+c_{k}q^{d(j_{k}-1)}-1, \\
& \vdots \\
& \left\lfloor\frac{m-1}{q^{dj_k}}\right\rfloor=c_{1}q^{d(j_1-j_k)}+c_{2}q^{d(j_2-j_k)}+\cdots+c_{k}-1, \\
& \vdots \\
& \left\lfloor\frac{m-1}{q^{dj_1}}\right\rfloor=c_{1}-1.
\end{align*}
Then $v_P(g!)  = \sum_{ j \geq 1}\left\lfloor\frac{m-1}{q^{dj}}\right\rfloor = e-j_{1} < e$.
This completes the proof.
\end{proof}

By Proposition~\ref{prop:comp2}, we directly obtain the following result.

\begin{corollary}
 Suppose that $P \in \A$ is an irreducible polynomial of degree $d$ and $e\leq q^{d}$ is a positive integer.
  Then
 \begin{equation*}
S(P^e)= a_{i_0} t^{d} + a_{i_1} t^{d+1} + \cdots + a_{i_k} t^{d+k},
\end{equation*}
where $e=\sum_{j=0}^{k}i_{j}q^{j} $ is the
$q$-adic expansion of $e$.
\end{corollary}

With some more effort we can estimate how the size of a polynomial changes after an action of $S$.
We in fact only need to consider reducible polynomials.

\begin{proposition}  \label{prop:delta-Kf}
Given a polynomial $f \in \A$ with $\deg f \ge 1$,
suppose that $f$ is reducible and $f \ne b(t+c)^2$ for any $b, c \in \F_q$.
Then, we have
$$
\delta(S(f)) \le \frac{\delta(f)}{q},
$$
where the equality holds if and only if $q=2$ or $3$, $f=t^3$.
\end{proposition}

\begin{proof}
Without loss of generality, we can assume that $f$ is monic.
When $f$ has at least two distinct monic irreducible factors,
by Proposition~\ref{prop:basic} (2) and Proposition~\ref{prop:comp1},
we immediately have
\begin{equation*}
\delta(S(f)) < \frac{\delta(f)}{q}.
\end{equation*}
So, it remains to consider the following two cases:
\begin{itemize}
\item $f=P^e, e \ge 2$ for a monic irreducible polynomial $P$ with $\deg P \ge 2$;

\item $f=P^e, e \ge 3$ for a monic linear polynomial $P$.
\end{itemize}

We first handle the first case.
That is,  we assume that $f=P^e, e \ge 2$ for a monic irreducible polynomial $P$ with $\deg P \ge 2$.

Let $d= \deg P$, and define $b_j = \frac{q^{dj}-1}{q^d - 1}, j \in \N^*$.
As before, $e$ can be uniquely written as
\begin{equation}  \label{eq:ecb}
e = c_{1} b_{j_1} + c_{2} b_{j_2} + \cdots + c_{k} b_{j_k},
\end{equation}
where $ j_1>j_2>\cdots>j_k>0 $
and $1\leq c_i < q^{d}$ for $i = 1,2,\ldots,k-1, 1\leq c_k\leq q^{d}$.
By Proposition~\ref{prop:comp2}, we have
\begin{equation}  \label{eq:Kcq}
\delta(S(f)) = c_{1} q^{dj_1} + c_{2} q^{dj_2} + \cdots + c_{k} q^{dj_k} \le q^{d(j_1 +1)}.
\end{equation}

If  $j_1 \ge 2$, then (using $d \ge 2$)
$$
e \ge  c_{1} b_{j_1} \ge b_{j_1} \ge 1 + 2^2 + \cdots + 2^{2(j_1 -1)} \ge j_1 + 3 ,
$$
and thus
$$
\frac{\delta(f)}{q^d} > \frac{q^{de}}{q^d} = q^{d(e -1)} \ge q^{d(j_1 +2)} > q^{d(j_1+1)} \ge \delta(S(f)).
$$

If $j_1 =1$,  then $e =c_1b_1 = c_1 \le q^d$ and $\delta(S(f)) = eq^{d}$, and so for $e \ge 3$,
$$
\frac{\delta(f)}{q^d} > q^{d(e-1)} \ge q^{2d} \ge eq^d = \delta(S(f)).
$$
For $e=2$,
$$
\frac{\delta(f)}{q} > \frac{q^{2d}}{q} = q^{2d-1} \ge 2q^d = \delta(S(f)).
$$
This completes the proof for the first case.

We now handle the second case.
That is, we assume that $f=P^e, e \ge 3$ for a monic linear polynomial $P$.
This means that in \eqref{eq:ecb} and \eqref{eq:Kcq} $d=1$.

If  $j_1 \ge 3$, then
$$
e \ge c_{1} b_{j_1} \ge b_{j_1} \ge 1 + 2 + \cdots + 2^{j_1 -1} \ge  j_1 + 4 ,
$$
and so
$$
\frac{\delta(f)}{q} \ge q^{e-1} \ge q^{j_1 + 3}  > q^{j_1 + 1} \ge \delta(S(f)).
$$

If $j_1 =2$, then for $e \ge 5$, we already have $e \ge j_1 +3$, and so
$$
\frac{\delta(f)}{q} \ge q^{e-1} \ge q^{j_1+2}   > q^{j_1 + 1} \ge \delta(S(f)).
$$
For $e=4$, we have either $q=2, e=b_1 + b_2, \delta(S(f)) =6$ or $q=3, e=b_2,\delta(S(f))=9$,
and then we still obtain
$$
\frac{\delta(f)}{q} \ge q^3  >  \delta(S(f)).
$$
For $e=3$, we must have $q=2, e= b_2, \delta(S(f)) =4$,
and so
$$
\frac{\delta(f)}{q}  \ge \frac{2^3}{2}  = 4 = \delta(S(f)),
$$
where the equality holds if and only if $f=t^3$.

If $j_1 =1$, then $e=c_1b_1 =c_1 \le q, \delta(S(f)) = eq$, and thus (using $e \ge 3$)
$$
\frac{\delta(f)}{q}  \ge    q^{e-1} \ge q^2 \ge eq = \delta(S(f)),
$$
where  the equalities hold if and only if $q=3, f=t^3$.
This completes the proof.
\end{proof}

In the above proof, we in fact have proved the following result.

\begin{corollary}
For any irreducible polynomial $P \in \A$ with $\deg P \ge 2$ and any integer $e \ge 3$, we have
$$
\delta(S(P^e)) < \frac{\delta(P^e)}{q^{\deg P}}.
$$
\end{corollary}

\subsection{Values of the Smarandache function}

Here, we consider the value set and the inverse image sets of the Smarandache function $S$.

\begin{proposition}\label{prop:val-range}
$ S(\A)= t\A $.
\end{proposition}
\begin{proof}
By Propositions~\ref{prop:comp1} and \ref{prop:comp2}, it is easy to see that $S(\A)\subseteq t\A$
(note that $S(b)=0$ for any $b \in \F_q$).
On the other hand, for any $f=a_{i_1}t + a_{i_2} t^2 + \cdots + a_{i_k} t^{k} \in t\A$,
we take  $e=i_{1}b_{1}+i_{2}b_{2}+\cdots+i_{k}b_{k}$,
where $b_{j}=\frac{q^{j}-1}{q-1}, j\in\N^{*}$.
 Then by Proposition~\ref{prop:comp2}, we have $S(t^{e})=f$, and so $t\A\subseteq S(\A)$.
 Thus $ S(\A)= t\A $.
\end{proof}

We have seen that the Smarandache function $S$ is not injective; see Proposition~\ref{prop:basic} (1).
For any non-zero polynomial $ f\in t\A$, denote by $S^{-1}(f)$ the inverse image set of $f$.
We now determine all the powers of irreducible polynomials contained in $S^{-1}(f)$.

\begin{proposition}\label{prop:inverse}
 Given a non-zero polynomial $f\in t\A$ and an integer $d\in\N^*$,
suppose that $q^d \mid \delta(f)$.
Then, $\delta(f)$ is uniquely represented as
\begin{equation*}
\delta(f)=c_{1}q^{dj_{1}}+c_{2}q^{dj_{2}}+\cdots+c_{k}q^{dj_{k}},
\end{equation*}
with $j_1>j_2>\cdots>j_k>0,1\leq c_i < q^{d}$ for $i =1,2,\ldots,k$.
Put $b_{j}=\frac{q^{dj}-1}{q^{d}-1}, j\in\N^*$
and
\begin{equation*}
e_{0}=c_{1}b_{j_{1}}+c_{2}b_{j_{2}}+\cdots+c_{k}b_{j_{k}}.
\end{equation*}
 Then $S^{-1}(f)$ contains the subset
$$
\{P^e: \, \textrm{$P\in\A$ is irreducible of degree $d$},\ e\in[ e_{0}-(j_{k}-1),e_{0}]\cap\N \}.
$$
In particular, when exhausting all the positive integers $d$ satisfying $q^d \mid \delta(f)$,
we obtain all the powers of irreducible polynomials contained in $S^{-1}(f)$.
\end{proposition}

\begin{proof}
Suppose that $P\in\A$ is an irreducible polynomial of degree $d$.
By Proposition~\ref{prop:comp2}, we directly have $\delta(S(P^{e_{0}}))=\delta(f)$, and so $S(P^{e_{0}})=f$.
When $j_{k}\geq 2$ and $e\in[ e_{0}-(j_{k}-1),e_{0})\cap\N $, without loss of generality, we take $e=e_{0}-i, 1\leq i \leq j_{k}-1$,
then $e$ is uniquely represented in the form:
\begin{align*}
e=c_{1}b_{j_{1}}+\cdots   & + c_{k-1} b_{j_{k-1}} + (c_{k}-1)b_{j_{k}} \\
& +(q^{d}-1)b_{j_{k}-1}+\cdots+(q^{d}-1)b_{j_{k}-(i-1)}+q^{d}b_{j_{k}-i}.
\end{align*}
By Proposition~\ref{prop:comp2}, we have $S(P^{e})=f$.
This in fact completes the proof.
\end{proof}

From Proposition~\ref{prop:inverse}, one can guess that the Smarandache function $S$ is not an increasing function.
We confirm this by the following result.

\begin{proposition}
For any irreducible polynomials $P, Q \in \A$ with $ \deg Q > 1 + \deg P$,
there exist positive integers $e_{1}$ and $e_{2}$ such that
$P^{e_{1}}>Q^{e_{2}}$ but $S(P^{e_{1}})<S(Q^{e_{2}})$.
\end{proposition}

\begin{proof}
For simplicity, denote $d_{1} = \deg P$ and $d_{2} = \deg Q$, and put $b_{j}=\frac{q^{jd_{1}}-1}{q^{d_{1}}-1},j\in\N^* $.
Since $1\leq d_{1}<d_{2}$, by the division algorithm, there exist $k,r\in\N$ such that
$$
d_{2}= kd_{1}+r,
$$
where $k \geq 1$ and $ 0\leq r < d_{1}$.
By assumption, we have  $d_2 - d_1 \ge 2$.

We first assume $r \neq 0$.
Take $e_{1}= b_1 + q^{r-1}b_{k}$ and $e_{2}=1$, then
$e_{1} \geq k +1$. So, using Proposition~\ref{prop:comp2} we have
$$
\delta(P^{e_{1}})\geq q^{d_{1}e_{1}}\geq q^{d_{1}(k +1)}=q^{d_{2}+d_{1}-r}\geq q^{d_{2}+1}>\delta(Q^{e_{2}})
$$
and
$$
\delta(S(P^{e_{1}}))= q^{d_1} + q^{r-1}q^{kd_{1}}=q^{d_1} + q^{d_2 - 1} < q^{d_{2}} = \delta(S(Q^{e_{2}})).
$$
Hence, $P^{e_{1}}>Q^{e_{2}}$ but $S(P^{e_{1}})<S(Q^{e_{2}})$.

We now assume  $r=0$.  Then $k \geq 2$.
 We take $e_{1}= b_1+b_{k - 1} + b_{k}$ and $e_{2}=2$, then $e_{1} \geq 2k + 1$.
So, using Proposition~\ref{prop:comp2} we deduce that
$$
\delta(P^{e_{1}})\geq q^{d_{1}e_{1}}\geq q^{d_{1}(2k +1)}=q^{2d_{2}+d_{1}}\geq q^{2d_{2}+1}>\delta(Q^{e_{2}})
$$
and
$$
\delta(S(P^{e_{1}}))=q^{d_{1}} + q^{d_{1}(k-1)} + q^{kd_{1}}=q^{d_{1}} + q^{d_{2}-d_{1}} + q^{d_{2}}< 2q^{d_{2}} = \delta(S(Q^{e_{2}})).
$$
Hence, $P^{e_{1}}>Q^{e_{2}}$ but $S(P^{e_{1}})<S(Q^{e_{2}})$.
This completes the proof.
\end{proof}

We remark that by Proposition~\ref{prop:comp2}, for any irreducible polynomials $P, Q \in \A$ and any positive integer $e$,
if $\deg P=\deg Q$, then $S(P^{e})=S(Q^{e})$.

\subsection{Fixed points}

For any $f\in \A$, if $S(f)=f$, then we call $f$ a fixed point of $S$.
We first determine the fixed points of the Smarandache function $S$.

\begin{proposition}
Given a non-zero polynomial $f \in \A$,
$ f$ is a fixed point of the Smarandache function $S$ if and only if
\begin{equation*}
\begin{split}
f
& = \left\{\begin{array}{ll}
t,  & \textrm{if $q>2$,}\\
\textrm{$t$ or $t^2$},  & \textrm{if $q=2$.}
\end{array}
\right.
\end{split}
\end{equation*}
\end{proposition}

\begin{proof}
If $ f $ is a fixed point, then $S(f)= f$,
and by Proposition~\ref{prop:basic} (1), (2) and Proposition~\ref{prop:comp1},
we must have $f=t^{e},e\in\N^*$.
So, by definition we obtain the desired result.
Indeed, by the definition of factorial (Definition~\ref{def:factorial}),
we have that $t^e \mid t^{e-1}!$ if $e>2$;
and  if $q>2$, then $t^2 \mid (a_2 t)!$.
\end{proof}

We remark that in the integer case all the prime numbers are fixed points of the Smarandache function.

For any integer $n \ge 1$, let $S^{(n)}$ be the $n$-th iteration of $S$.
It is easy to see that for any $f \in \A$ with $\deg f \ge 1$
there exists some integer $n$ such that $S^{(n)}(f)$ is a fixed point of $S$.
We now estimate the number of iterations, which can be viewed as the distance to fixed points.

\begin{proposition}  \label{prop:dist}
For any $f \in \A$ with $\deg f \ge 1$, there exists a positive integer $n \le 1 + \deg f$
such that $S^{(n)}(f)$ is a fixed point of $S$.
\end{proposition}

\begin{proof}
We first note that
by Proposition~\ref{prop:val-range}, for any polynomial $f \in \A$ with $\deg f \ge 1$,
 we have $S(f) \in t\A$, and so,
if $\deg S(f) \ge 2$, then $S(f)$ must be a reducible polynomial.

Now, given $f \in \A$ with $\deg f \ge 1$, if $\deg S(f) \ge 3$,
then $S(f)$ satisfies the condition in Proposition~\ref{prop:delta-Kf}, and so
$$
\delta(S^{(2)}(f)) \le \frac{\delta(S(f))}{q}.
$$
If again $\deg S^{(2)}(f) \ge 3$, we have
$$
\delta(S^{(3)}(f)) \le \frac{\delta(S^{(2)}(f))}{q} \le  \frac{\delta(S(f))}{q^2}.
$$
This process stops when we reach $\deg S^{(j)}(f) \le 2$ for some integer $j$.
So, this integer $j$ satisfies
$$
\delta(S^{(j)}(f)) < q^3.
$$
This automatically holds if
$$
\frac{\delta(S(f))}{q^{j-1}} < q^3,
$$
which, together with $\delta(S(f)) \le q^{\deg f}$ by Proposition~\ref{prop:basic} (2), is implied in
$$
q^{\deg f - j +1} < q^3.
$$
Thus $j \ge \deg f - 1$.

If $\deg S^{(j)}(f) = 2$, since $S^{(j)}(f)$ is reducible, we have that $S^{(j+2)}(f)$ must be a fixed point of $S$.

Then, we always have that $S^{(j+2)}(f)$ is a fixed point of $S$.
Hence, for $n = 1 + \deg f$, $S^{(n)}(f)$ is a fixed point of $S$.
\end{proof}

We remark that the upper bound in Proposition~\ref{prop:dist} is optimal.
Because when $q\ge 3$,
$$
S^{(3)}(f) = S^{(2)}(t^2) = S(a_2 t) =t
$$
for any irreducible polynomial $f \in \A$ of degree $2$.

\section{Analogue of Erd{\H o}s's problem}
\label{sec:Erdos}

In this section, we establish an analogue of Erd{\H o}s's problem.

For any non-constant polynomial $f \in \A$, let $\cP(f)$ be the maximal monic irreducible factor of $f$.
Following Erd{\H o}s's problem, for any integer $n \ge 1$, we define the subset of $\A$:
$$
\cT(n) = \{\textrm{monic }f \in \A:\, \deg f=n, S(f) \ne t^{\deg \cP(f)} \}.
$$
One should note that $S(\cP(f)) = t^{\deg \cP(f)}$.

Using the strategy in \cite{Akbik}, which is in fact a classical approach
by considering the number of distinct prime factors and the maximal prime factor,
we establish the following analogue of Erd{\H o}s's problem.

\begin{theorem}  \label{thm:Erdos}
For the sets $\cT(n)$, we have
\begin{equation*}
\begin{split}
|\cT(n)| <
&  \left\{\begin{array}{ll}
2^{n+2} \exp(-\sqrt{n}/3),  & \textrm{if $q=2$ and $n \ge 3249$,}\\
4q^n \exp(-\sqrt{n}/2),  & \textrm{if $q \ge 3$ and $n \ge \max\{13000, \log q\}$.}
\end{array}
\right.
\end{split}
\end{equation*}

\end{theorem}

Theorem~\ref{thm:Erdos} implies that for almost all polynomials $f \in \A$, $S(f)=t^d$,
where $d$ is the maximal degree of the irreducible factors of $f$.

Recall the sets $\cT_1(n,r), \cT_2(n,r), \cT_3(n,r)$ defined in Section~\ref{sec:count}.
To prove Theorem~\ref{thm:Erdos}, we need one more preliminary result.

\begin{lemma} \label{lem:S4}
For any integer $n \ge 1$ and any real $r \ge 1$, define
$$
\cT_4(n,r) = \cT(n) \setminus (\cT_1(n,r) \cup \cT_2(n,r) \cup \cT_3(n,r)).
$$
Then, for any $f \in \cT_4(n,r)$ we have
$$
\deg \cP(f) < D + \frac{\log D}{\log q}.
$$
\end{lemma}

\begin{proof}
For any $f \in \cT_4(n,r)$, we have $S(f) \ne t^{\deg \cP(f)}$, which implies that  $f \nmid t^{\deg \cP(f)}!$.
So, there exists a monic irreducible polynomial $P$
such that $P \mid f$ and $v_P(f) > v_P(t^{\deg \cP(f)}!)$.
Note that $v_P(t^{\deg \cP(f)}!) \ge 1$ by definition. So, we have $v_P(f) \ge 2$.
Then, in view of $f \not\in \cT_2(n,r)$ and $f \not\in \cT_3(n,r)$, we must have
$\deg P \le D$ and $v_P(f) < D$.

Hence, using Lemma~\ref{lem:val} we obtain
\begin{align*}
D > v_P(f) > v_P(t^{\deg \cP(f)}!) \ge \frac{q^{\deg \cP(f)}}{q^{\deg P}}
\ge \frac{q^{\deg \cP(f)}}{q^D},
\end{align*}
which gives the desired result.
\end{proof}

We are now ready to prove Theorem~\ref{thm:Erdos}.

\begin{proof}[Proof of Theorem~\ref{thm:Erdos}]
For any $f \in \cT_4(n,r)$, we have $f \not\in \cT_1(n,r) \cup \cT_2(n,r) \cup \cT_3(n,r)$.
So, $\omega(f) \le B = 3r \log\log q^n$, and also,
if $P^e, e \ge 1,$ is any positive power of a monic irreducible polynomial $P$ such that $P^e \mid f$,
then we only have two cases:
\begin{itemize}
\item[(i)] $\deg P \le D, e < D$,

 \item[(ii)] $\deg P > D, e = 1$,
\end{itemize}
where $D=2r\log\log q^n$.
Case (i) yields at most $\lfloor Dq^D \rfloor$ positive powers of monic irreducible polynomials (using Lemma~\ref{lem:count-irre}).
For Case (ii), since $\deg P \le \deg \cP(f) < D + \log D / \log q$ by Lemma~\ref{lem:S4},
it also gives at most $\lfloor Dq^D \rfloor$ positive powers of monic irreducible polynomials.
Hence, the number of possible powers of monic irreducible polynomials which divides an $f \in \cT_4(n,r)$
is not greater than $2Dq^D$.
However, such an $f$ is the product of at most $B = 3r \log\log q^n$ distinct powers of monic irreducible polynomials.
Hence, we have
\begin{equation} \label{eq:S4}
\begin{split}
|\cT_4(n,r)|
& \le (2Dq^D)^{B} \\
& = (4r \log\log q^n)^{3r \log\log q^n} \cdot q^{6 (r\log\log q^n)^2} \\
& \le q^{7(r\log\log q^n)^2}
\end{split}
\end{equation}
when $r \log\log q^n \ge 19$.
Then, assuming moreover $n \ge 3$ and $r \ge 2$ and using Lemmas~\ref{lem:S1}, \ref{lem:S2} and \ref{lem:S3}, we obtain
\begin{equation} \label{eq:S}
\begin{split}
|\cT(n)| & \le |\cT_1(n,r)| + |\cT_2(n,r)| + |\cT_3(n,r)| + |\cT_4(n,r)| \\
& < \frac{3q^n}{(\log q^n)^r} + q^{7(r\log\log q^n)^2}.
\end{split}
\end{equation}
Now, choosing
$$
r =  \frac{\sqrt{n}}{3\log\log q^n},
$$
we obtain
$$
|\cT(n)| < 4q^n \exp(-\sqrt{n}/3)
$$
when $n \ge 3249$ and $n \ge (6\log\log q^n)^2$.
Here the condition on $n$ comes from $\sqrt{n}/3 = r \log\log q^n \ge 19$ and $r \ge 2$.

If $q = 2$, we have
$$
|\cT(n)| < 2^{n+2} \exp(-\sqrt{n}/3)
$$
when $n \ge 3249$.

We now assume $q \ge 3$.
In this case, assuming $n \ge 4, r \ge 3$ and $r\log \log q^n \ge 19$ and using Lemmas~\ref{lem:S1}, \ref{lem:S2} and \ref{lem:S3},
we can choose $B = 2r \log \log q^n$ and $D = r\log \log q^n$,
and then \eqref{eq:S4} becomes
\begin{equation*}
\begin{split}
|\cT_4(n,r)|
& \le (2Dq^D)^{B} \\
& = (2r \log\log q^n)^{2r \log\log q^n} \cdot q^{2 (r\log\log q^n)^2} \\
& \le q^{3 (r\log\log q^n)^2}.
\end{split}
\end{equation*}
So, \eqref{eq:S} becomes
$$
|\cT(n)|  < \frac{3q^n}{(\log q^n)^r} + q^{3(r\log\log q^n)^2}.
$$
Now, choosing
$$
r =  \frac{\sqrt{n}}{2\log\log q^n},
$$
we obtain
$$
|\cT(n)| < 4q^n \exp(-\sqrt{n}/2)
$$
when $n \ge 1444$ and $n \ge (6\log\log q^n)^2$.
Here the condition on $n$ comes from $\sqrt{n}/2 = r\log \log q^n \ge 19$ and $r \ge 3$.
Finally, it is easy to see that if
$$
n \ge \max\{13000, \log q\},
$$
then the condition $n \ge (6\log\log q^n)^2$ holds.
This completes the proof.
\end{proof}

\section*{Acknowledgement}
The authors thank the referee for careful reading of the paper and valuable comments.
They also thank George Martin for valuable comments.
The first author was supported by the National Science Foundation of China Grant No. 11526119
and the Scientific Research Foundation of Qufu Normal University No. BSQD20130139,
and the second author was supported by the Australian Research Council Grant DE190100888.

\end{document}